\documentclass[12pt]{article}
\usepackage{amsmath,amssymb,amsthm,comment}

\newtheorem{theorem}{Theorem}[section]
\newtheorem{thmy}{Theorem}

\newtheorem{corollary}[theorem]{Corollary}

\newcommand{\dd}{\displaystyle }

\def\barr{\begin{array}}
\def\earr{\end{array}}

\title{On a generalization of the Reidemeister spectrum of a finite group}
\author{Marius T\u arn\u auceanu}
\date{September 27, 2026}

\begin{document}

\maketitle

\begin{abstract}
Based on a recent generalization of the Reidemeister spectrum of a finite group \cite{1}, in this short note we extend the main result of our previous paper \cite{2}.
\end{abstract}

{\small
\noindent
{\bf MSC 2020\,:} Primary 20D45, 20E45; Secondary 20E22. 

\noindent
{\bf Key words\,:} ZM-group, (bi-)twisted conjugacy, Reidemeister number, (ge\-ne\-ra\-li\-zed) Reidemeister spectrum.} 

\section{Introduction}

Let $G$ be a group and ${\rm Aut}(G)$ be the group of automorphisms of $G$. The concept of conjugacy in $G$ is naturally generalized to the so-called
\textit{twisted conjugacy}: given $\varphi\in {\rm Aut}(G)$, two elements $x,y\in G$ are called \textit{$\varphi$-conjugate} if there exists $g\in G$ such that $x=gy\varphi(g)^{-1}$. We define the \textit{Reidemeister number} $R(\varphi)$ of $\varphi$ as the number of equivalence classes with respect to this relation. Also, we define the \textit{Reidemeister spectrum} of $G$ to be the set ${\rm Spec}_R(G)=\{R(\varphi);\varphi\in {\rm Aut}(G)\}$. 

Twisted conjugacy in $G$ can be also generalized to \textit{bi-twisted conjugacy}: given $\varphi,\psi\in {\rm Aut}(G)$, two elements $x,y\in G$ are said to be \textit{$(\varphi,\psi)$-conjugate} if $x=\psi(g)y\varphi(g)^{-1}$ for some $g\in G$ (see \cite{1}). Then the number of equivalence classes with respect to this relation is called the \textit{Reidemeister number} of $(\varphi,\psi)$ and is denoted by $R(\varphi,\psi)$. Also, we call the \textit{generalized Reidemeister spectrum} of $G$ the set ${\rm Spec}'_R(G)=\{R(\varphi,\psi);\varphi,\psi\in {\rm Aut}(G)\}$.

The Reidemeister spectrum of a ZM-group has been determined in \cite{2}. In what follows, we will extend this result by computing the generalized Reidemeister spectrum of such a group. Recall that a ZM-group has the following presentation
\begin{equation}
{\rm ZM}(m,n,r)=\langle a, b \mid a^m = b^n = 1,
\hspace{1mm}b^{-1} a b = a^r\rangle, \nonumber
\end{equation}where the triple $(m,n,r)$ satisfies the conditions
\begin{equation}
(m,n)=(m,r-1)=1 \mbox{ and } r^n \equiv 1 \hspace{1mm}({\rm mod}\hspace{1mm}m).\nonumber
\end{equation}Note that $|{\rm ZM}(m,n,r)|=mn$ and $Z({\rm ZM}(m,n,r))=\langle b^{o_m(r)}\rangle$, where
\begin{equation}
o_m(r)={\rm min}\{k\in\mathbb{N}^*; r^k\equiv 1 \hspace{1mm}({\rm mod} \hspace{1mm}m)\}\,\,\footnote{\,\,For simplicity of notation, throughout the paper we will use $d$ instead of $o_m(r)$.}\nonumber
\end{equation}is the multiplicative order of $r$ modulo $m$. We also recall that each automorphism of ${\rm ZM}(m,n,r)$ is of type
\begin{equation}
f_{(x_1,x_2,x)}: b^ua^v\mapsto b^{xu}a^{x_1v+x_2[u]_r},\, u,v\geq 0,\nonumber
\end{equation}where the triple $(x_1,x_2,x)$ satisfies the conditions
\begin{equation}
0\leq x_1,x_2<m,\, (x_1,m)=1,\, 0\leq x<n \mbox{ and } x\equiv 1 \hspace{1mm}({\rm mod}\hspace{1mm}d).\nonumber
\end{equation}We will denote by $A_{(m,n,r)}$ the set of all these triples. Moreover, for every $u,v\in\mathbb{N}$ and $(x_1,x_2,x)\in A_{(m,n,r)}$ we will denote
\begin{itemize}
\item[-] $[u]_r=\left\{\barr{lll}
    \!\!1+r+\cdots+r^{u-1},&u>0\\
    \!\!0,&u=0\earr\right.$,
\item[-] $i_{(x_1,x_2)}(u,v)=x_1v+x_2[u]_r$,
\item[-] $j_{(x_1,x_2,x)}(u,v)=(m,[xu]_r,i_{(x_1,x_2)}(u,v))$.
\end{itemize}
\smallskip

Our main result is stated as follows.\vspace{1mm}

\begin{theorem}
We have ${\rm Spec}'_R({\rm ZM}(m,n,r))=$
\begin{equation}
\left\{\!\dd\frac{1}{m}\dd\sum_{v=0}^{m-1}\!\!\!\!\!\!\!\!\!\sum_{^{\,\,\,\,\,\,\,\,\,\,\,\,\,\,\,\,\,\,\,\,\,\,\,\,\,0\leq u\leq n-1}_{^{\,\,\,\,\,\,\,\,\,\,\,\,\,\,\,\,\,\,\,\,\,\,\frac{n}{(n,x-y)}\mid u}_{j_{(y_1,y_2,y)}(u,v)\mid i_{(x_1,x_2)}(u,v)}}}\!\!\!\!\!\!\!\!\!\!\!\frac{(m,[yu]_r)}{o_{\frac{(m,[yu]_r)}{j_{(y_1,y_2,y)}(u,v)}}(r)}; (x_1,x_2,x),(y_1,y_2,y)\in A_{(m,n,r)}\!\right\}\!.\nonumber\vspace{2mm}
\end{equation}
\end{theorem}

Note that in the particular case when $n$ is prime we reobtain Corollary 1.3 of \cite{2}.

\begin{corollary}
If $n$ is prime, then we have
\begin{equation}
{\rm Spec}'_R({\rm ZM}(m,n,r))={\rm Spec}_R({\rm ZM}(m,n,r))=\left\{n-1+\dd\frac{S}{n}\right\}\!,\nonumber
\end{equation}where
\begin{equation}
S=\dd\sum_{u=0}^{n-1}\,(m,[u]_r).\nonumber
\end{equation}
\end{corollary}

\section{Proofs of the main results}

\begin{proof}[Proof of Theorem 1.1.] 
Let $\varphi=\varphi_{(x_1,x_2,x)},\psi=\psi_{(y_1,y_2,y)}\in{\rm Aut}({\rm ZM}(m,n,r))$. In order to compute $R(\varphi,\psi)$, we will apply Burnside's lemma to the action of ${\rm ZM}(m,n,r)$ on ${\rm ZM}(m,n,r)$ given by
\begin{equation}
g\circ x=\psi(g)x\varphi(g)^{-1}, \forall\, g,x\in{\rm ZM}(m,n,r).\nonumber
\end{equation}Let $g=b^ua^v$ and $x=b^{\alpha}a^{\beta}$, where $0\leq u,\alpha<n$ and $0\leq v,\beta<m$. We have
\begin{align*}
&Fix(g)=\{x\in{\rm ZM}(m,n,r); \varphi(g)=x^{-1}\psi(g)x\}\\
&=\{b^{\alpha}a^{\beta}\in{\rm ZM}(m,n,r); b^{xu}a^{x_1v+x_2[u]_r}=b^{yu}a^{r^{\alpha}(y_1v+y_2[u]_r)-\beta(r^{yu}-1)}\}\\
&=\{b^{\alpha}a^{\beta}\in{\rm ZM}(m,n,r); b^{xu}=b^{yu} \mbox{ and } a^{x_1v+x_2[u]_r}=a^{r^{\alpha}(y_1v+y_2[u]_r)-\beta(r^{yu}-1)}\}.\nonumber
\end{align*}We remark that
\begin{equation}
b^{xu}=b^{yu}\Leftrightarrow n\mid (x-y)u\Leftrightarrow\frac{n}{(n,x-y)}\mid u\nonumber
\end{equation}and
\begin{equation}
a^{x_1v+x_2[u]_r}\!\!=a^{r^{\alpha}(y_1v+y_2[u]_r)-\beta(r^{yu}-1)}\Leftrightarrow m\mid (r^{\alpha}y_1-x_1)v-\beta(r^{yu}-1)+(r^{\alpha}y_2-x_2)[u]_r.\nonumber
\end{equation}For fixed $u,v$ and $\alpha$, the last congruence has solutions $\beta$ if and only if 
\begin{equation}
(m,[yu]_r)\mid (r^{\alpha}y_1-x_1)v+(r^{\alpha}y_2-x_2)[u]_r,
\end{equation}since $r^{yu}-1=[yu]_r(r-1)$ and $(m,r-1)=1$. Moreover, if it has solutions, then there are exactly $(m,[yu]_r)$ solutions. Note that (1) is equivalent with
\begin{equation}
(m,[yu]_r)\mid r^{\alpha}i_{(y_1,y_2)}(u,v)-i_{(x_1,x_2)}(u,v).
\end{equation}For fixed $u,v$, (2) has solutions $\alpha$ if and only if 
\begin{equation}
(m,[yu]_r,i_{(y_1,y_2)}(u,v))\mid i_{(x_1,x_2)}(u,v),\nonumber
\end{equation}which means
\begin{equation}
j_{(y_1,y_2,y)}(u,v)\mid i_{(x_1,x_2)}(u,v).\nonumber
\end{equation}Also, it is easy to see that (2) has the same number of solutions $\alpha$ as
\begin{equation}
\frac{(m,[yu]_r)}{j_{(y_1,y_2,y)}(u,v)}\mid r^{\alpha}-1,\nonumber
\end{equation}that is as\vspace{-2mm}
\begin{equation}
o_{\frac{(m,[yu]_r)}{j_{(y_1,y_2,y)}(u,v)}}(r)\mid\alpha.
\end{equation}Obviously, (3) has $\frac{n}{o_{\frac{(m,[yu]_r)}{j_{(y_1,y_2,y)}(u,v)}}(r)}$ solutions and so we obtain
\begin{align*}
&\hspace{40mm}R(\varphi,\psi)=\dd\frac{1}{mn}\dd\sum_{g\in{\rm ZM}(m,n,r)}\text{card}(Fix(g))\\
&=\dd\frac{1}{mn}\dd\sum_{(u,v)}\text{card}(\{(\alpha,\beta); b^{xu}=b^{yu} \mbox{ and } a^{x_1v+x_2[u]_r}=a^{r^{\alpha}(y_1v+y_2[u]_r)-\beta(r^{yu}-1)}\})\\
&=\dd\frac{1}{mn}\sum_{v=0}^{m-1}\sum_{^{\,0\leq u\leq n-1}_{\frac{n}{(n,x-y)}\mid u}}\sum_{\alpha=0}^{n-1}\text{card}(\{0\leq\beta<m; m\!\mid\! (r^{\alpha}y_1{-}x_1)v{-}\beta(r^{yu}{-}1){+}(r^{\alpha}y_2{-}x_2)[u]_r\})\\
&=\dd\frac{1}{mn}\sum_{v=0}^{m-1}\sum_{^{\,0\leq u\leq n-1}_{\frac{n}{(n,x-y)}\mid u}}\!\sum_{^{\,\,\,\,\,\,\,\,\,\,\,\,\,\,\,\,\,\,\,\,\,\,\,\,\,\,\,\,\,\,\,\,\,\,\,\,\,\,0\leq \alpha\leq n-1}_{\,\,\,(m,[yu]_r)\mid (r^{\alpha}y_1-x_1)v+(r^{\alpha}y_2-x_2)[u]_r}}\!\!\!(m,[yu]_r)\\
&=\dd\frac{1}{mn}\sum_{v=0}^{m-1}\sum_{^{\,0\leq u\leq n-1}_{\frac{n}{(n,x-y)}\mid u}}\!\!\!(m,[yu]_r)\text{card}(\{0\leq\alpha<n; (m,[yu]_r)\mid (r^{\alpha}y_1{-}x_1)v{+}(r^{\alpha}y_2{-}x_2)[u]_r\})\\
&=\dd\frac{1}{mn}\dd\sum_{v=0}^{m-1}\!\!\!\!\!\!\!\!\!\sum_{^{\,\,\,\,\,\,\,\,\,\,\,\,\,\,\,\,\,\,\,\,\,\,\,\,\,0\leq u\leq n-1}_{^{\,\,\,\,\,\,\,\,\,\,\,\,\,\,\,\,\,\,\,\,\,\,\frac{n}{(n,x-y)}\mid u}_{j_{(y_1,y_2,y)}(u,v)\mid i_{(x_1,x_2)}(u,v)}}}\!\!\!\!\!\!\!\!\!\!\!(m,[yu]_r)\,\frac{n}{o_{\frac{(m,[yu]_r)}{j_{(y_1,y_2,y)}(u,v)}}(r)}\\
&=\dd\frac{1}{m}\dd\sum_{v=0}^{m-1}\!\!\!\!\!\!\!\!\!\sum_{^{\,\,\,\,\,\,\,\,\,\,\,\,\,\,\,\,\,\,\,\,\,\,\,\,\,0\leq u\leq n-1}_{^{\,\,\,\,\,\,\,\,\,\,\,\,\,\,\,\,\,\,\,\,\,\,\frac{n}{(n,x-y)}\mid u}_{j_{(y_1,y_2,y)}(u,v)\mid i_{(x_1,x_2)}(u,v)}}}\!\!\!\!\!\!\!\!\!\!\!\frac{(m,[yu]_r)}{o_{\frac{(m,[yu]_r)}{j_{(y_1,y_2,y)}(u,v)}}(r)}\,,\nonumber
\end{align*}as desired.
\end{proof}

We observe that if in the above proof we take $\psi=1_{{\rm ZM}(m,n,r)}$, then $y_1=y=1$ and $y_2=0$, which lead to
\begin{equation}
j_{(y_1,y_2,y)}(u,v)=(m,[u]_r,v)\mid i_{(x_1,x_2)}(u,v),\nonumber
\end{equation}for all $0\leq u<n$ and $0\leq v<m$. Thus we get the formula in Theorem 1.1 of \cite{2}.

\begin{proof}[Proof of Corollary 1.2.] 
If $n$ is prime, then $d=n$ and so the conditions $0\leq x,y<n$ and $x,y\equiv 1\, ({\rm mod}\, d)$ imply $x=y=1$ for any $\varphi=\varphi_{(x_1,x_2,x)},\psi=\psi_{(y_1,y_2,y)}\in{\rm Aut}({\rm ZM}(m,n,r))$. Thus, we have 
\begin{equation}
\frac{n}{(n,x-y)}=1\mid u, \mbox{ for all } 0\leq u<n.\nonumber 
\end{equation}We also obtain that
\begin{equation}
j_{(y_1,y_2,1)}(u,v)=(m,[u]_r,y_1v+y_2[u]_r)=(m,[u]_r,y_1v)=(m,[u]_r,v)\nonumber
\end{equation}divides $([u]_r,v)$ and consequently it divides
\begin{equation}
i_{(x_1,x_2)}(u,v)=x_1v+x_2[u]_r,\nonumber
\end{equation}for all $0\leq v<m$ and $0\leq u<n$. Then
\begin{equation}
R(\varphi,\psi)=\dd\frac{1}{m}\dd\sum_{v=0}^{m-1}\sum_{u=0}^{n-1}\frac{(m,[u]_r)}{o_{\frac{(m,[u]_r)}{(m,[u]_r,v)}}(r)}=n-1+\dd\frac{S}{n}\nonumber
\end{equation}by Corollary 1.3 of \cite{2}. This completes the proof.
\end{proof}

\bigskip\noindent{\bf Acknowledgements.} The author is grateful to the reviewer for remarks which improved the previous version of the paper.

\newpage

\vspace*{3ex}\small

\hfill
\begin{minipage}[t]{5cm}
Marius T\u arn\u auceanu \\
Faculty of  Mathematics \\
``Al.I. Cuza'' University \\
Ia\c si, Romania \\
e-mail: {\tt tarnauc@uaic.ro}
\end{minipage}

\end{document}